\documentclass[a4paper,draft,reqno,12pt]{amsart}
\usepackage[english]{babel}
\usepackage{amsmath}
\usepackage{amssymb}
\usepackage{amscd}
\usepackage{amsthm}
\usepackage{euscript}
\newtheorem{theor}{Theorem}[section]
\newtheorem{prop}[theor]{Proposition}
\newtheorem{lem}[theor]{Lemma}

\newtheorem{conj}[theor]{Conjecture}
\newtheorem{cor}[theor]{Corollary}
\theoremstyle{definition}
\newtheorem{de}[theor]{Definition}
\newtheorem{ex}[theor]{Example}
\theoremstyle{remark}
\newtheorem {re}[theor]{Remark}

\DeclareMathOperator{\Aut}{Aut}

\def\Ker{{\rm Ker}\,}
\def\Im{{\rm Im}\,}

\def\HH{{\mathbb H}}
\def\SAut{{\mathrm{SAut}}}

\def\GG{{\mathbb G}}

\def\CC{{\mathbb C}}
\def\KK{{\mathbb K}}
\def\LL{{\mathbb L}}
\def\TT{{\mathbb T}}
\def\ZZ{{\mathbb Z}}

\def\SS{{\mathbb S}}

\def\CCC{\mathcal{C}}
\def\OO{\mathcal{O}}

\def\ZZ{\mathbb{Z}}

\def\VV{\mathbb{V}}
\def\SSS{\mathcal{S}}
\def\LLL{\mathcal{L}}
\def\Ga{\mathbb{G}_a}
\def\Gm{\mathbb{G}_m}

\begin{document}

\title{Orbits of automorphism group of trinomial hypersurfaces}
\author{Sergey Gaifullin}
\address{Moscow Center for Fundamental and Applied Mathematics, Moscow, Russia; \linebreak
Lomonosov Moscow State University, Faculty of Mechanics and Mathematics, Department of Higher Algebra, Leninskie Gory 1, Moscow, 119991 Russia; \linebreak and \linebreak
National Research University Higher School of Economics, Faculty of Computer Science, Pokrovsky Boulevard 11, Moscow, 109028, Russia}
\email{sgayf@yandex.ru}
\thanks{The first author was supported by RSF grant 20-71-00109.}

\author{Georgiy Shirinkin}
\address{Lomonosov Moscow State University, Faculty of Mechanics and Mathematics, Department of Higher Algebra, Leninskie Gory 1, Moscow, 119991 Russia}
\email{goshir59rus@gmail.com}



\maketitle

\begin{abstract}
Trinomial hypersurfaces form a natural class of affine algebraic varieties closely connected with varieties admitting a torus action of complexity one. We investigate orbits of the automorphism group on these hypersurfaces. We prove that each nonrigid trinomial variety has finite number of orbits. We investigate singular orbits and this gives us description of all orbits for some classes of flexible trinomial hypersurfaces. Also we obtain a description of orbits for a class of hypersurfaces having a unique variable with power one in the equation. 
\end{abstract}

\section{Introduction}

Let $\KK$ be an algebraically closed field of characteristic zero. Assume $X$ be an affine algebraic variety over $\KK$. We are interested in orbits of the natural action of the group of all regular automorphisms $\Aut(X)$ on $X$. To investigate this action it is natural to consider algebraic subgroups of $\Aut(X)$. Every connected algebraic group is generated by one-parameter subgroups, each of them is isomorphic to additive goup $(\KK,+)$ or multiplicative group $(\KK^\times,\cdot)$. We call such subgroups $\Ga$ and $\Gm$-subgroups respectively. So, it is natural to consider the subgroup $\Aut_{alg}$, generated by all connected algebraic subgroups in $\Aut(X)$. Also we consider the subgroup $\SAut(X)$ of special automorphisms, i.e. the subgroup, generated by all $\Ga$ subgroups. This subgroups are normal subgroups of $\Aut(X)$. Therefore, each automorphism permutes their orbits. So, our plan of describing $\Aut(X)$-orbits is firstly to describe $\Aut_{alg}$-orbits, and then to check if two $\Aut_{alg}$-orbits can be glued by an automorphism. 

An automorphism can not take a regular point to a singular one. So the biggest possible $\Aut(X)$- orbit is the regular locus $X^{reg}$. A situation when $X^{reg}$ is an $\SAut(X)$-orbit has a special interest. In this situation the variety $X$ is called {\it flexible}. It was proved in \cite{AFKKZ} that if $X$ is a flexible variety of dimension $\geq 2$, then the action of $\SAut(X)$ on $X^{reg}$ is infinitely transitive, i.e. each $m$-tuple of pairwise distinguish points can be simultaneously taken to each another $m$-tuple of pairwise distinguish points by an automorphism $\varphi\in \SAut(X)$. In the opposite situation, when $X$ does not admit any nontrivial $\Ga$-subgroups,  the variety $X$ is called {\it rigid}. For each affine algebraic variety we can put into correspondence its Makar-Limanov invariant $ML(X)$, i.e. the intersection of kernels of all locally nilpotent derivations. For flexible varieties $ML(X)=\{0\}$, for rigid ones $ML(X)=\KK[X]$. In other cases $ML(X)$ gives some information about $\SAut(X)$ since $ML(X)$ is the ring of $\SAut(X)$-invariant functions.
 
To deal with one-parameter subgroups we use some standard technique. An action of one-parameter multiplicative subgroup corresponds to a $\ZZ$-grading of the algebra of regular functions $\KK[X]$. Also it corresponds to a semisimple derivation of $\KK[X]$. Actions of one-parameter additive subgroup correspond to locally nilpotent derivations. If we have both a grading and a derivation, it is natural to consider the decomposition of this derivation onto the sum of homogeneous components. In this paper we elaborate some technique in case, when for a fixed $\ZZ$-filtration all LNDs has nonnegative degrees. 

In this paper we describe $\Aut(X)$-orbits for some classes of trinomial hypersurfaces. Let $n=n_0+n_1+n_2$, where $n_1$ and $n_2$ are positive integers and $n_0$ is a nonnegative integer. By a trinomial hypersurface we mean a variety given in $\KK^n$ by the following equation
$$
T_{01}^{l_{01}}\ldots T_{0n_0}^{l_{0n_0}}+T_{11}^{l_{11}}\ldots T_{1n_1}^{l_{1n_1}}+T_{21}^{l_{21}}\ldots T_{2n_2}^{l_{2n_2}}=0.
$$
If $n_0=0$, we mean that $T_{01}^{l_{01}}\ldots T_{0n_0}^{l_{0n_0}}=1$, see \cite[Construction~1.1]{HW}. We say that a trinomial hypersurface is {\it with a free term}, if $n_0=0$,  and {\it without a free term} otherwise.This class of varieties is interesting since every irreducible, normal, rational affine variety $X$ with only constant invertible functions, finitely generated divisor class group and an effective algebraic torus action of complexity one can be obtained as a quotient of a trinomial variety via action of a diagonalizable group, see \cite[Corrolary~1.9]{HW}. Here trinomial variety is a variety given by a system of trinomial equations of special form. Moreover this trinomial variety is the total coordinate space of $X$. Background in total coordinate spaces and Cox rings you can find in \cite{ADHL}. The simplest case of a trinomial variety is a trinomial hypersurface. It is easy to see that a trinomial hypersurface admits an action of a torus $\TT$ of complexity one. 

In~\cite{Ar} all factorial rigid trinomial hypersurfaces without a free term are described. In~\cite{Ga} there are criterium for arbitrary trinomial hypersurfaces to be rigid and a sufficient condition to be flexible. In~\cite{AG} orbits of $\Aut(X)$ on a rigid trinomial hypersurface without free term are described. Each $\Aut(X)$-orbit on this case can be obtained by gluing of a finite number of $\TT$-orbits. Therefore, there are infinitely many $\Aut(X)$-orbits. We prove that in all other cases the number of $\Aut(X)$-orbits is finite, see Theorem~\ref{ro} and Corrolary~\ref{fo}. Also we investigate singular orbits. We prove that for a wide class of trinomial hypersurfaces singular $\Aut_{alg}$-orbits are just $\TT$-orbits, see Theorem~\ref{sing}. This class consists of all hypersurfaces such that all powers $l_{ij}\geq 2$. 

The main technical problem in description of $\Aut(X)$-orbits is to compute the Makar-Limanov invariant $ML(X)$. We can do it in the case of trinomial variety of the form 
$$
\VV(xy_1^{a_1}\ldots y_m^{a_m}+z_1^{b_1}\ldots z_p^{b_p}+s_1^{c_1}\ldots s_m^{c_m}), b_j, c_k\geq 2.
$$
For this class of trinomial hypersurfaces we have a description of $\Aut(X)$-orbits, see Theorem~\ref{maintheor}. Note that for this class of trinomial hypersurfaces is not contained in the class for which we have described singular $\Aut_{alg}(X)$-orbits. But for this class it turns out that singular $\Aut_{alg}(X)$-orbits also coinside with $\TT$-orbits. 
We have a conjecture about $ML(X)$ in more general case. In assumption that it is true, we describe $\Aut(X)_{alg}$-orbits, see Theorem~\ref{dd}. 

The first author is a Young Russian Mathematics award winner and would like to thank its sponsors and jury.

\section{Preliminareis}

\subsection{Derivations}

Let $A$ be a commutative associative algebra over $\KK$.
A linear mapping $\partial\colon A\rightarrow A$ is called a {\it derivation} if it satisfies the Leibniz rule $\partial(ab)=a\partial(b)+b\partial(a)$ for all $a,b\in A$.
A derivation is {\it locally nilpotent} (LND) if for every $a\in A$ there is a positive integer~$n$ such that $\partial^n(a)=0$.
A derivation is {\it semisimple} if there exists a basis of $A$ consisting of $\partial$-semi-invariants. Recall that $a\in A$ is a $\partial$-semi-invariant if $\partial(a)=\lambda a$ for some $\lambda\in\KK$.

If we have an algebraic action of the additive group $(\KK,+)$ on $A$, we obtain a one-parameter $\Ga$-subgroup in $\Aut(A)$. Then its tangent vector at unity is an LND. Exponential mapping defines a bijection between LNDs and elements in $\Ga$-subgroups of $\mathrm{Aut}(A)$. Similarly, if  we have $\Gm$-subgroup of~$\Aut(A)$, then its tangent vector at unity is a semisimple derivation.

Let $F$ be an abelian group.
An algebra $A$ is called {\it $F$-graded} if 
$$A=\bigoplus_{f\in F}A_f,$$ 
and $A_fA_g\subset A_{f+g}$ for all $f$ and $g$ in $F$.
A derivation $\partial\colon A\rightarrow A$ is {\it $F$-homogeneous of degree $f_0\in F$}, if for every $a\in A_f$ we have $\partial(a)\in A_{f+f_0}$.

A derivation $\partial\colon A\rightarrow A$ is called {\it locally finite}, if any element $a\in A$ is contained in a $\partial$-invariant linear subspace $V\subset A$ of finite dimension.
It is easy to see that all semisimple derivations and all LNDs are locally finite.

Let A be a finitely generated $\ZZ$-graded algebra and $\partial$ be a derivation of $A$. Then $\partial$ can be decomposed  into  $\partial=\sum_{i=l}^k\partial_i$, where $\partial_i$ is a homogeneous derivation of degree $i$. Further when we write $\partial=\sum_{i=l}^k\partial_i$, we assume that $\partial_l\neq 0$ and $\partial_k\neq 0$.

\begin{lem}\label{fl}(See \cite{Re} for (1) and \cite[Lemma~3.1] {FZ} for (2))
Let $A$ be a finitely generated $\ZZ$-graded algebra. Assume $\partial\colon A\rightarrow A$ is a derivation. We have $\partial=\sum_{i=l}^k\partial_i$, where $\partial_i$ is a homogeneous derivation of degree $i$. Then

1) if $\partial$ is an LND then $\partial_l$ and $\partial_k$ are LNDs.

2) if $\partial$ is locally finite then if $l\neq 0$, $\partial_l$ is an LND, and if $k\neq 0$, $\partial_k$ is an LND.
\end{lem}

Let $X$ be an affine algebraic variety and $A=\KK[X]$ be the algebra of regular functions on~$X$. Then $\ZZ^n$-gradings on $\KK[X]$ are in bijection with actions of $n$-dimensional algebraic torus $T=(\KK^\times)^n$ on $X$. 
If we have a $\ZZ^n$-homogeneous derivation of $\KK[X]$, we call it {\it $T$-homogeneous derivation}. Each LND we can decompose onto a sum of $\mathbb{Z}^n$-homogeneous derivations. Lemma~\ref{fl} implies, that homogeneous component corresponding to vertexes of convex hull of degrees of nonzero summands are LNDs.

Let us use the notation $a\mid b$ if $a$ divides $b$.  

\begin{lem}\label{semisimp}
Assume $A$ is a domain. Let $\delta$ be a semisimple derivation, corresponding to a subgroup $\Lambda~\cong~\KK^\times$ of $\Aut(A)$. Suppose for some $f\in A$ we have $f\mid\delta(f)$. Then $f$ is $\Lambda$-homogeneous, and hence, for every $\varphi\in\Lambda$, there exists $\lambda\in\KK^\times$ such that  $\varphi(f)=\lambda f$.
\end{lem}
\begin{proof}
We have $\delta(f)=fg$. Let us consider $\ZZ$-grading corresponding to $\Lambda$. Let $f~=~\sum_{i=l}^k f_i$ and $g=\sum_{j=p}^q g_j$ be decompositions onto homogeneous components. Then $\delta(f)~=~\sum_{u=l+p}^{k+q}\delta(f)_u$. Since $A$ is a domain, $\delta(f)_{l+p}\neq 0$ and $\delta(f)_{k+q}\neq 0$. But $\delta$ acts on $A_i$ by multiplying by $i$, that is $\delta(f)~=~\sum_{i=l}^k if_i$. So, we have $p=q=0$. Hence, $g\in A_0$. Then $\delta(f)_i=g f_i=i f_i$. Hence, $f$ is homogeneous of some degree $i$. Thus, $\varphi(f)=t^i f$ for some $t\in \KK^\times$. 
\end{proof}
\subsection{Flexible, rigid and semi-rigid affine varieties}

\begin{de}
A variety is called {\it rigid} if its algebra of regular functions does not admit any LNDs. 
\end{de}

By \cite[Theorem~2.1]{AG}, see also~\cite[Section~3]{FZ} the automorphism group of a rigid variety has unique maximal torus. This gives a powerful tool for investigating automorphisms of rigid varieties.

\begin{de}
A variety $X$ is called {\it semi-rigid} if kernels of all LNDs of $\KK[X]$ coinside. 
\end{de}

For semi-rigid varieties $ML(X)$ has transendense degree equal to $\dim X-1$, therefore, dimensions of $\SAut(X)$-orbits not bigger then one. Hence, all $\SAut(X)$-orbits are lines or points. 

\begin{de}
A variety $X$ is called {\it flexible} if for every regular point $x\in X^{reg}$, the tangent space $T_{x}X$ is spanned by tangent vectors to 
$\Ga$-orbits for various regular $\Ga$-subgroups in $\Aut(X)$.

\end{de}

\begin{de}
The subgroup of $\Aut(X)$ generated by all $\Ga$-subgroups is called {\it subgroup of special automorphisms}. We denote it by $\mathrm{SAut}(X)$.
\end{de}

Let a group $G$ act on a set $X$. This action is called {\it $m$-transitive} if for every two $m$-tuples $(a_1,\ldots, a_m)$ and $(b_1, \ldots, b_m)$, where $a_i\neq a_j$ and $b_i\neq b_j$ if $i\neq j$, there is an element $g$ in $G$ such that for all $i$ we have $g\cdot a_i=b_i$.
If an action is $m$-transitive for every positive integer $m$, then it is called {\it infinitely transitive}.

One of the main results on flexible varieties is the following.
\begin{prop}\label{aaa}\cite[Theorem~0.1]{AFKKZ}
For an irreducible affine variety $X$ of dimension $\geq 2$, the following
conditions are equivalent.

(i) The group $\mathrm{SAut}(X)$ acts transitively on $X^{\mathrm{reg}}$,

(ii) The group $\mathrm{SAut}(X)$ acts infinitely transitively on $X^{\mathrm{reg}}$,

(iii) The variety $X$ is flexible.
\end{prop}

\subsection{Trinomial hypersurfaces}\label{TH}

Let us consider an affine space $\KK^n$. We fix a partition $n=n_0+n_1+n_2$, where $n_0$ is a nonnegative integer and $n_1,n_2$ are positive integers. The coordinates on $\KK^n$ we denote by $T_{ij}$, $i\in\{0,1,2\}$, $1\leq j\leq n_i$.

\begin{de}
A {\it trinomial hypersurface} is a subvariety $X$ in $\KK^n$ given by the unique equation
$$T_{01}^{l_{01}}T_{02}^{l_{02}}\ldots T_{0n_0}^{l_{0n_0}}+T_{11}^{l_{11}}T_{12}^{l_{12}}\ldots T_{1n_1}^{l_{1n_1}}+T_{21}^{l_{21}}T_{22}^{l_{22}}\ldots T_{2n_2}^{l_{2n_2}}=0,$$
where $l_{ij}$ are positive integers.
If $n_0=0$, we obtain the variety 
$$\VV\left(1+T_{11}^{l_{11}}\ldots T_{1n_1}^{l_{1n_1}}+T_{21}^{l_{21}}\ldots T_{2n_2}^{l_{2n_2}}\right).$$ 
We call such a variety a {\it trinomial hypersurface with a free term}. 
If $n_0\neq 0$, then $X$ is a {\it trinomial hypersurface without a free term}. 
\end{de}

By \cite[Theorem~1.2(i)]{HW} every trinomial hypersurface is irreducible and normal.
We denote the monomial $T_{i1}^{l_{i1}}T_{i2}^{l_{i2}}\ldots T_{in_i}^{l_{in_i}}$ by $T_i^{l_i}$. So, trinomial hypersurface has the form $\VV\left(T_0^{l_0}+T_1^{l_1}+T_2^{l_2}\right).$

Let us consider the linear action of $n$-dimensional torus on the polynomial algebra $\KK[T_{ij}]$ given by multiplying each coordinate $T_{ij}$ by $t_{ij}$. The subgroup which stabilizes the linear span $\langle T_0^{l_0}+T_1^{l_1}+T_2^{l_2}\rangle$ we denote by $\HH$. This is a diagonalizable group, i.e. the direct product of its neutral component $\TT$, which is a $(n-1)$-dimensional torus and a commutative finite group $\GG$. Then $\TT$-action is an action of an algebraic torus on $X$ of complexity one. Sometimes $X$ has natural symmetries given by permutations of variables. Let us denote by $\SS(X)$ the group of permutations of variables, that stabilize the equation of $X$. We call this group by {\it group of symmetries of $X$}.

If $n_i=1$ and $l_{i1}=1$, then $X\cong \KK^{n-1}$. 
If for all $i$ we have $l_{i1}n_i>1$, then $X$ is factorial if and only if the numbers 
$$d_i = \gcd(l_{i1} , . . . , l_{in_i} ),\qquad i = 0, 1, 2$$ 
are pairwise coprime, see \cite[Theorem~1.1 (ii)]{HH}. A trinomial hypersurface with a free term is factorial if and only if $d_1=d_2=1$.
A factorial trinomial hypersurface without a free term is rigid if and only if all $l_{ij}\geq 2$, see \cite[Theorem~1]{Ar}.
In \cite[Theorem~2]{Ga} is obtained the following generalization of  \cite[Theorem~1]{Ar}.
\begin{prop}\label{rt}
A trinomial hypersurface $X=\VV\left(T_0^{l_0}+T_1^{l_1}+T_2^{l_2}\right)$ is not rigid if and only if one of the following holds

1) there exist $i\in\{0, 1, 2\}$ and $a\in \{1, 2, \ldots, n_i\}$ such that $l_{ia}=1$,

2) $n_0\neq 0$ and there exist $i\neq j\in\{0, 1, 2\}$ and $a\in \{1, 2, \ldots, n_i\}$, $b\in \{1, 2, \ldots, n_j\}$ such that $l_{ia}=l_{jb}=2$ and for all 
$u\in \{1, 2, \ldots, n_i\}$, $v\in \{1, 2, \ldots, n_j\}$ the numbers $l_{iu}$ and $l_{jv}$ are even.
\end{prop}

In~\cite[Theorem 5.5]{AG} the automorphism group of a rigid trinomial hypersurface is described. It is isomorphic to $\SS(X)\rightthreetimes\HH$.  

In all other cases we can consider some locally nilpotent derivations given by the following formulas.

1) Let $X=\VV\left(T_0^{l_0}+T_1^{l_1}+T_{21}T_{22}^{l_{22}}\ldots T_{2n_2}^{l_{2n_2}}\right)$.  For every $i\in\{0,1\}$ and  $j\in\{1,2,\ldots, n_i\}$
we can define an LND $\gamma_{ij}$ of $\KK[X]$ by
$$
\gamma_{ij}(T_{21})=-\frac{\partial T_i}{\partial T_{ij}},\qquad
\gamma_{ij}(T_{ij})=\frac{\partial T_2}{\partial T_{21}},\qquad
\gamma_{ij}(T_{pq})=0\ \text{for all other pairs}\ (p,q).
$$

2) Let  
$$X=\VV\left(T_{01}^2T_{02}^{2m_{02}}\ldots T_{0n_0}^{2m_{0n_0}}-T_{11}^2T_{12}^{2m_{12}}\ldots T_{1n_1}^{2m_{1n_1}}-T_2^{l_2}\right). $$ 
For each $1\leq i\leq n_2$ we have two LNDs $\delta_{i+}$ and $\delta_{i-}$ of $\KK[X]$ given by

\begin{equation*}
\begin{cases}
\delta_{i+}(T_{01})=\frac{\partial T_2}{\partial T_{2i}}T_{12}^{m_{12}}\ldots T_{1n_1}^{m_{1n_1}};\\
\delta_{i+}(T_{11})=-\frac{\partial T_2}{\partial T_{2i}}T_{02}^{m_{02}}\ldots T_{0n_0}^{m_{0n_0}};\\
\delta_{i+}(T_{2i})=2T_{02}^{m_{02}}\ldots T_{0n_0}^{m_{0n_0}}T_{12}^{m_{12}}\ldots T_{1n_1}^{m_{1n_1}}\left(
T_{01}T_{02}^{m_{02}}\ldots T_{0n_0}^{m_{0n_0}}+T_{11}T_{12}^{m_{12}}\ldots T_{1n_1}^{m_{1n_1}}\right);\\
\delta_{i+}(T_{jk})=0, \forall (j,k)\notin\{ (0,1), (1,1), (2,i)\}.
\end{cases}
\end{equation*}
\begin{equation*}
\begin{cases}
\delta_{i-}(T_{01})=\frac{\partial T_2}{\partial T_{2i}}T_{12}^{m_{12}}\ldots T_{1n_1}^{m_{1n_1}};\\
\delta_{i-}(T_{11})=\frac{\partial T_2}{\partial T_{2i}}T_{02}^{m_{02}}\ldots T_{0n_0}^{m_{0n_0}};\\
\delta_{i-}(T_{2i})=2T_{02}^{m_{02}}\ldots T_{0n_0}^{m_{0n_0}}T_{12}^{m_{12}}\ldots T_{1n_1}^{m_{1n_1}}\left(
T_{01}T_{02}^{m_{02}}\ldots T_{0n_0}^{m_{0n_0}}-T_{11}T_{12}^{m_{12}}\ldots T_{1n_1}^{m_{1n_1}}\right);\\
\delta_{i-}(T_{jk})=0, \forall (j,k)\notin\{ (0,1), (1,1), (2,i)\}.
\end{cases}
\end{equation*}

For some trinomial hypersurfaces these LNDs provide flexibility of $X$. In \cite{Ga} the following types $H_1-H_5$ of trinomial hypersurfaces were introduced.

\ 

\begin{tabular}{|c|c|}
\hline
{\tiny\ }&{\tiny\ }\\
$H_1$&
$\VV\left(T_0^{l_0}+T_1^{l_1}+T_{21}T_{22}\ldots T_{2n_2}\right)$\\
{\tiny\ }&{\tiny\ }\\
\hline
{\tiny\ }&{\tiny\ }\\
$H_2$&$\VV\left(T_{01}^2T_{02}^2\ldots T_{0n_0}^2+T_{11}^2T_{12}^2\ldots T_{1n_1}^2+T_2^{l_2}\right)$\\
{\tiny\ }&{\tiny\ }\\
\hline
{\tiny\ }&{\tiny\ }\\
$H_3$&$\VV\left(T_0^{l_0}+T_{11}T_{12}^{l_{12}}\ldots T_{1n_1}^{l_{1n_1}}+T_{21}T_{22}^{l_{22}}\ldots T_{2n_2}^{l_{2n_2}}\right)$\\
{\tiny\ }&{\tiny\ }\\
\hline
{\tiny\ }&{\tiny\ }\\
$H_4$&$
\VV\left(T_{01}T_{02}^{l_{02}}\ldots T_{0n_0}^{l_{0n_0}}+T_{11}^2T_{12}^{2m_{12}}\ldots T_{1n_1}^{2m_{1n_1}}+T_{21}^2T_{22}^{2m_{22}}\ldots T_{2n_2}^{2m_{2n_2}}\right)$\\
{\tiny\ }&{\tiny\ }\\
\hline
{\tiny\ }&{\tiny\ }\\
$H_5$&$\VV\left(T_{01}^2T_{02}^{2m_{02}}\ldots T_{0n_0}^{2m_{0n_0}}+T_{11}^2T_{12}^{2m_{12}}
\ldots T_{1n_1}^{2m_{1n_1}}+T_{21}^2T_{22}^{2m_{22}}\ldots T_{2n_2}^{2m_{2n_2}}\right)$\\
{\tiny\ }&{\tiny\ }\\
\hline
\end{tabular}

\begin{prop}(\cite[Theorem~4]{Ga}) Trinomiaal hypersurfaces of types $H_1-H_5$ are flexible.
\end{prop}

The variety, which is not rigid and is not flexible, we call {\it intermediate}. In \cite{Ga} for some types of trinomial hypersurfaces it is proved that they are intermediuate and for some of them $ML(X)$ was computed, see \cite[Proposition 2 and 3]{Ga}.

\section{Partial order on an algebra}

The contest of this section partially can be find in ~\cite[section~3]{G}. But since all proofs are very short and we use other terms, we give here complete proofs. 

Let $A$ be an affine $\KK$-domain. Let us introduce a partial order $\succeq$ on $A$.

\begin{de}
We say that $a\succeq b$ if for every LND $\partial$ of $A$, if $\partial(a)=0$, then $\partial(b)=0$. 
\end{de}
It is easy to see that this order is transitive, i.e. $a\succeq b$ and $b\succeq c$ imply $a\succeq c$, and reflexive $a\succeq a$. But it is not antisymmetric, i.e. $a\succeq b$ and $b\succeq a$ do not imply $a=b$. Moreover, if $f$ and $g$ are two nonconstant polynomials in $\KK[t]$, then $f(a)\succeq g(a)$ and $g(a)\succeq f(a)$.  
Also this order respects multiplication, i.e. $a \succeq b$ and $c\succeq d$ imply $ac\succeq bd$. The set $M(a)=\{b\in A\mid a\succeq b\}$ is a factorially closed subalgebra in $A$.
The set of minimal elements coincide with $ML(X)$. Indeed, $ML(X)=\{a\in A| b\succeq a\text{ for all } b\in A\}$. For each algebra $A$ the set of maximal elements, i.e. elements $a$ such that $M(a)=A$, is nonempty. For example, if we take a system  of generators $a_1,\ldots, a_s$ of $A$ and consider $a=a_1\ldots a_s$, then $a$ is a maximal element.

Let us remind the construction of $m$-suspension, see~\cite{G}. Let $X$ be an affine variety. Fix $f\in \KK[X]\setminus \KK$. Then $Y=\mathrm{Susp}(X,f,l_1,\ldots, l_m)$ is a subvariety in
$X\times \KK^n$ given by the equation $y_1^{l_1}\ldots y_m^{l_m}=f$, where $y_1,\ldots, y_m$ are coordinates on $\KK^n$. We can consider a $\ZZ$-grading $\xi_{ij}$ on $\KK[Y]$ such that $\deg y_i=-k_j$, $\deg y_j=k_i$, degrees of all the other $y_p$ are zero, and degrees of all $g\in\KK[X]$ are zero. 

\begin{lem} \label{nap}
If $y_1$ is a maximal element of $\KK[Y]$, then $\xi_{1j}$  admits homogeneous LND only with  positive degrees.
\end{lem}
\begin{proof}
Suppose $\partial$ is a homogeneous LND with a nonpositive  $\xi_{1j}$-degree. Then $\xi_{1j}$-degree of $\partial(y_1)$ is negative. That is $y_1\mid \partial(y_1)$. That is $\partial(y_1)=0$, see~\cite[Principle~5]{Fr}. Since $y_1$ is a maximal element, $\partial(y_1)=0$ implies $\partial=0$. 
\end{proof}

\begin{re}
It is not true that if $y_i\succeq y_j$, then $\xi_{ij}$  admits homogeneous LND only with  positive degrees. And this is the main problem in attempts compute $ML(X)$ for an arbitrary trinomial variety. Indeed, for any homogeneous LND $\delta$ such that $\delta(y_i)=0$, we can consider the LND $\delta'=y_i^N\delta$. We have $\deg \delta'=\deg \delta-Nk_j$. If $N$ is enough large, then $\deg \delta'$ is negative. 
\end{re}

\begin{lem}\label{leml}
If $y_1$ is a maximal element of $\KK[Y]$, then 

a) for every locally finite derivation $\delta$ we have $y_j\mid \delta(y_j)$ for all $j>1$;

b) $y_j\in ML(Y)$ for all $j>1$;

c) $y_j$ is a semi-invariant with respect to each $\KK^\times$-action on $\KK[Y]$.
\end{lem}
\begin{proof}
a) Let $\delta =\sum_{i=l}^k \delta_i$ be the decomposition of $\delta$ onto the sum of $\xi_{1j}$-homogeneous derivations. If $l<0$, by Lemma~\ref{fl} we have $\delta_l$ is a homogeneous LND with negative degree. This contradicts to Lemma~\ref{nap}. Therefore, $l\geq 0$. Hence, for each $l\leq i\leq k$ we have $\deg \delta_i(y_j)>0$. So, $y_j\mid \delta_i(y_j)$. Hence, $y_j\mid \delta(y_j)$ .

b) If $\delta$ is locally nilpotent, then by \cite[Principle~5]{Fr}  $y_j\mid \delta(y_j)$ implies $\delta(y_j)=0$.  

c) The assertion of c) follows from a) and Lemma~\ref{semisimp}.
\end{proof}

\begin{cor}\label{coco}
Assume $y_1$ is a maximal element of $\KK[Y]$ and $ML(Y)\subseteq \KK[y_2,\ldots, y_m]$. Then $ML(Y)=\KK[y_2,\ldots, y_m]$ and for every $\varphi\in \Aut(X)$ and for every $j>1$ there exists $i>1$ and  $\lambda\in\KK^\times$ such that $\varphi(y_j)=\lambda y_i$.
\end{cor}
\begin{proof}
By Lemma~\ref{leml}(b) $y_j\in ML(Y)$ for all $j>1$. Since $ML(Y)\subseteq \KK[y_2,\ldots, y_m]$, we have $ML(Y)=\KK[y_2,\ldots, y_m]$. Thus, $y_j$ for $j>1$ are indecomposible elements of $ML(Y)$ semi-invariant under each $\KK^\times$ action, see Lemma~\ref{leml}(c). But each $\ZZ$-grading $\xi_{ij}$ corresponds to a $\KK^\times$-action. Element is semi-invariant with respect to this action if and only if it is homogeneous with respect to each $\xi_{ij}$. It is easy to see that indecomposable element of $\KK[y_2,\ldots, y_m]$ homogeneous with respect to all $\xi_{ij}$ equals $\lambda y_i$, $i>1$. Since this conditions do not chenge under an automorphism, we have $\varphi(y_j)=\lambda y_i$.
\end{proof}

We have some assertions with conditions that $y_1$ is a maximal element. So we need a sufficient condition for $y_1$ to be maximal.
\begin{lem} \label{cm}
If there exists a positive integer $n$ such that  $Z=\mathrm{Susp}(X,f,nl_1,l_2,\ldots, l_m)$ is rigid, then the element $y_1$ is maximal.
\end{lem}
\begin{proof}
If $y_1$ is not maximal, then there exists a nonzero LND $\partial$ such that $\partial(y_1)=0$. For each $g\in\KK[X][y_2,\ldots, y_m]$ we have $\partial(g)=h(y_1)$, where $h\in \KK[X][y_2,\ldots, y_m][t]$. Then $\overline{\partial}(g)=h(y_1^n)$, $\overline{\partial}(y_1)=0$ is a nonzero LND of $Z$.
\end{proof}

\section{$\Aut(X)_{alg}$-orbits on trinomial varieties}

Let $X$ be a rigid trinomial hypersurface. Then $\Aut(X)=\SS(X)\rightthreetimes \HH$. So, $\SAut(X)$-orbits are points, $\Aut_{alg}$-orbits coincide with $\TT$-orbits and each $\Aut(X)$-orbit is a finite union of $\TT$-orbits. Therefore, there are infinitely many $\Aut(X)$-orbits. 

Now let us consider a nonrigid trinomial hypersurface. Each derivation of $\KK[X]$ can be extended to a derivation of the field of rational functions $\KK(X)$. Denote by $\KK(X)^\TT$ the subfield of rational $\TT$-invariants. Recall that a homogeneous derivation $\delta$ of $\KK[X]$ is of vertical type if $\delta(\KK(X)^\TT)=0$ holds and of horizontal type otherwise. 
The folowing lemma is proved in~\cite{Ar}.
\begin{lem}(\cite[Lemma~2]{Ar}) Let $X$ be a trinomial hypersurface. Then $\KK[X]$ admits no nonzero homogeneous
locally nilpotent derivation of vertical type.
\end{lem}
Since $\TT$-action is of complexity one, it follows that $\Aut_{alg}(X)$-rational invariants are constants. So, since each $\Aut_{alg}(X)$-orbit is locally closed, see~\cite[Proposition~1.3]{AFKKZ}  we obtain the following corollary.
\begin{cor}
Let $X$ be a nonrigid trinomial hypersurface. Then there exists an open $\Aut_{alg}(X)$-orbit.
\end{cor}
Let us prove that the number of $\Aut_{alg}(X)$-orbits is finite.
\begin{theor}\label{ro}
Let $X$ be a nonrigid trinomial hypersurface. Then the number of $\Aut_{alg}(X)$-orbits is finite.
\end{theor}
\begin{proof}
Let $S$ be a subset of the set of variables $\{T_{ij}\}$.  Let us take $\alpha$ and $\beta$ such points that on $X$ that $T_{ij}(\alpha)=0\Leftrightarrow T_{ij}\in S$ and $T_{ij}(\beta)=0\Leftrightarrow T_{ij}\in S$. Note that $S$ can be chosen by finite number of variants.  

Case 1) $l_{21}=1$. 

Suppose all $T_{2j}, j\geq 2$ does not belong to $S$. Then,  there is $t\in \TT$ such that 
$T_{2j}(t\cdot \alpha)=T_{2j}(\beta)$. Applying composition $\varphi$ of $\exp(u_{ij}\gamma_{ij})$ we can do that on $\varphi(t\cdot\alpha)$ and $\beta$ all functions $T_{0i}$ and $T_{1i}$ give the same values. Since $\varphi$ does not change $T_{2j}, j\geq 2$, we have that  these functions are also coincide on $\varphi(t\cdot\alpha)$ and $\beta$. Therefore, $T_{21}$ also coincides on this points. That is $\varphi(t\cdot\alpha)=\beta$. Remark that this case follows from~\cite[Lemma~14]{Ga}.

Suppose there is $j\geq 2$ such that $T_{2j}\in S$. If both $T_1^{l_1}$ and $T_2^{l_2}$ have at least one variable in $S$, then $\alpha$ can be taken to $\beta$ by an element of $\TT$. Otherwise we can do by an element of $\TT$ that all functions except $T_{11}$ coincide on $t\cdot\alpha$ and $\beta$. Since $(T_1^{l_1}+T_2^{l_2})(\beta)=0$, there are finite opportunities for $T_{11}(\beta)$. Hence, there are finite number of $\Aut_{alg}(X)$-orbits with this $S$.

Case 2) $$X=\VV\left(T_{01}^2T_{02}^{2m_{02}}\ldots T_{0n_0}^{2m_{0n_0}}-T_{11}^2T_{12}^{2m_{12}}\ldots T_{1n_1}^{2m_{1n_1}}-T_2^{l_2}\right). $$ 
Suppose all $T_{0j}, j\geq 2$ and $T_{1j}, j\geq 2$ does not belong to $S$. Then $\alpha$ can be taken to $\beta$ by $\SAut(X)$, see~\cite[Lemma~15]{Ga}. If one or two monomials are zero we proceed by the same way as in Case 1.

\end{proof}
\begin{cor}\label{fo}
Let $X$ be a nonrigid trinomial hypersurface. Then the number of $\Aut(X)$-orbits is finite.
\end{cor}

\begin{re}
It is easy to see, that all semi-rigid trinomial hypersurfaces has the form $xy^a+z^b+1$ with  $a,b>1$. Indeed, in all other cases of nonrigid trinomial hypersurfaces we have at least two LNDs of the form $\gamma_{ij}$ or $\delta_{i\pm}$. Hypersurfaces of this form are partiqular cases of Danielewski surfaces. Their automorphism groups were described in~\cite{ML}. From this discription follows that this surfaces are indeed semi-rigid. See also Section~\ref{msms} for description of orbits of a class of trinomial surfaces containing this surfaces. From results of Section~\ref{msms} follows that we have two orbits: $\{y=0\}$ and $\{y\neq 0\}$.
\end{re}

Let $S$ be a subset of the set of variables $\{T_{ij}\}$. It can be subdivided onto union of $S_0\cup S_1\cup S_2$, where $S_a=S\cap\{ T_{aj}\}$  For each $S$ let us consider the subset $U(S)\subset X$ given by 
$$
U(S)=\{\alpha\in X\mid T_{ij}(\alpha)=0\Leftrightarrow T_{ij}\in S\}.
$$
If $S$ contains at least one variable from each monomial, one can see that $U(S)$ is a $\TT$-orbit. In particular, each $U(S)$ consisting of singular points is a $\TT$-orbit. 
Denote

$$X(S)=\bigcup_{S\subseteq P}U(P).$$  
It is easy to see that $X(S)$ is the closure of $U(S)$. I.e. $X(S)=\VV(I(S))$, where $I(S)$ is the ideal of $\KK[X]$, generated by $S$.
Each irreducible component of the singular locus $X^{sing}$ is a subset $X(S)$, where $S$ for each $i$ contains either a unique variable $T_{ij}$ with $l_{ij}\geq 2$, or two variables $T_{ij}$ and $T_{ik}$, where $l_{ij}=l_{ik}=1$.  

\begin{de}
Suppose 
$$T_i^{l_i}=T_{i1}\ldots T_{ik}T_{i,k+1}^{l_{i,k+1}}\ldots T_{in_i}^{l_{in_i}}, l_{ij}\geq 2.$$ 
Let $\varnothing\neq A_i\subseteq\{T_{i,k+1},\ldots ,T_{in_i}\}$. Put $M(A_i)=\{A_i,A_i\cup T_{ij}| j\leq k\}.$
We say that two sets $S=S_0\cup S_1\cup S_2$ and $P=P_0\cup P_1\cup P_2$ are {\it linked} if for each $i$ either $S_i=P_i$, or $S_i,P_i\in M(A_i)$ for some $A_i$
\end{de}
\begin{prop}\label{linn}
 If two sets $S$ and $P$ are not linked, then $N(S)\neq N(P)$ and $U(S)$ can not be glued with $U(P)$ by $\Aut_{alg}(X)$.
\end{prop}
\begin{proof}
Let us prove that each  irreducible component of  $X^{sing}$ is $\Aut_{alg}(X)$-invariant. The group $\Aut_{alg}(X)$ is generated by one-parameter subgroups isomorphic to $(\KK,+)$ and $(\KK^\times,\cdot)$. Suppose there exists a one-parameter subgroup $\Lambda$ such that one of its orbits $\Lambda\cdot \beta$ nontrivially intersects but is not contained in a component. Each automorphism from $\Lambda$ permutes irreducible components of $X^{sing}$. Then $\Lambda\cdot \beta$ is covered by all components $V_i$ such that $\Lambda\cdot \beta\nsubseteq V_{i}$. This contradicts to irreducibility of $\Lambda$. So, we have proved that a point in an irreducible component $V$ can not be moved by $\Aut_{alg}(X)$ in a point not in $V$. For each $U(S)\subseteq X^{sing}$ we denote by $N(S)$ the set of irreducible components of $X^{sing}$ containing $U(S)$. If $N(S)\neq N(P)$, then $U(S)$ and $U(P)$ are not contained in one $\Aut_{alg}(X)$-orbit. 

Denote by $D(S)$ intersection of all elements of $N(S)$. And consider
$$
L(S)=D(S)\setminus\bigcup_{D(P)\subsetneq D(S)} D(P).
$$
Then $L(S)$ is the union of $U(P)$, where $P$ is linked with $S$.
\end{proof}

\begin{cor}
If a $\TT$-orbit corresponds to such a set $S$ that there are no linked sets, then this $\TT$-orbit is $\Aut_{alg}(X)$-orbit.
\end{cor}

\begin{theor}\label{sing}
Assume all $l_{ij}\geq 2$. Then
all singular $\Aut_{alg}$-orbits of $X$ coincide with $\TT$-orbits.
\end{theor}
\begin{proof}
Since all  $l_{ij}\geq 2$, there are no linked sets.
\end{proof}

For factorial varieties conditions of Theorem~\ref{sing} imply rigidity of $X$. But non-factorial variety with these conditions can be even flexible. Varieties of type $H_5$ and $H_2\setminus H_4$ satisfy these conditions. For a flexible trinomial variety $X$ all regular points of $X$ form one $\SAut(X)$-orbit. So, description of singular $\Aut_{alg}(X)$-orbits gives a description of all $\Aut_{alg}(X)$-orbits. We obtain the following theorem.
\begin{theor}
For varieties of type $H_5$ and $H_2\setminus H_4$ all regular points form one $\Aut_{alg}(X)$-orbit and singular $\Aut_{alg}(X)$-orbits are $\TT$-orbits.
\end{theor}

We have proved that all singular $\TT$-orbits, which are not linked to any other, are $\Aut_{alg}(X)$-orbits.  But we do not have any examples of two linked $\TT$-orbits, that can be glued by $\Aut_{alg}(X)$. So, it is natural to state the following conjecture.

\begin{conj}
All singular $\Aut_{alg}$-orbits of a trinomial hypersurface coincide with $\TT$-orbits. 
\end{conj}

In the next section we prove this conjecture for a class of trinomial varieties of special form admitting linked $\TT$-orbits. 

\begin{ex}
Consider three trinomial hypersurfaces
$$
X_1=\VV(y^2+z^3+s^3)
$$
$$
X_2=\VV(xy^2+z^3+s^3)
$$
$$
X_3=\VV(x_1x_2y^2+z^3+s^3)
$$
In $X_1$  there are no linked $\TT$-orbits. $X_1$ has the unique singular point $(0,0,0)$, which is a $\TT$-orbit. The variety $X_2$ has irreducible singular locus $\{y=z=s=0\}$ consisting on two linked $\TT$-orbits: $\{(\lambda\neq 0,0,0,0)\}$ and $(0,0,0,0)$. Proposition~\ref{linn}
does not separate these two $\TT$-orbits. But Theorem~\ref{AA} will do it. The variety $X_3$ has two irreducible components of $X^{sing}$: $V_1=\{y=z=s=0\}$ and $V_2=\{x_1=x_2=z=s=0\}$. $X_3^{sing}$ consists of the following~5 $\TT$-orbits: $\OO_1=\{(0,0,0,0,0)\}$, $\OO_2=\{(\lambda\neq 0,0,0,0,0)\}$, $\OO_3=\{(0,\mu\neq0,0,0,0)\}$, $\OO_4=\{(\lambda\neq 0,\mu\neq0,0,0,0)\}$, $\OO_5=\{(0,0,\tau\neq 0, 0 ,0)\}$. Orbits $\OO_2,\OO_3$ and $\OO_4$ are linked because they are contained in $V_1\setminus V_2$. And we do not know if these orbits can be glued by $\Aut_{alg}(X)$.
\end{ex}

\section{Unique variable with power one}\label{msms}
In this section we describe $\Aut(X)$-orbits of a nonflexible trinomial hypersurface with the unique variable with power one in the trinom.

Let 
$$
X=\mathbb{V}(xy_1^{a_1}\ldots y_m^{a_m}+z_1^{b_1}\ldots z_p^{b_p}+s_1^{c_1}\ldots s_q^{c_q}),  a_i>1, m\geq 1, p\geq 1, q\geq 0.
$$
If $q=0$ we have a trinomial hypersurface with a free term.
Assume that $X$ is not of types $H_3$ and $H_4$.
It is convenient to have a new notations for LNDs $\gamma_{ij}$  introduced in Section~\ref{TH} in terms of these section.
$$D_i(x)=\frac{\partial (z_1^{b_1}\ldots z_p^{b_p})}{\partial z_i}, D_i(z_i)=-y_1^{a_1}\ldots y_m^{a_m}, D_i \text{ is zero on all other variables;}$$
$$E_j(x)=\frac{\partial (s_1^{c_1}\ldots s_q^{c_q})}{\partial s_j}, E_j(s_j)=-y_1^{a_1}\ldots y_m^{a_m}, E_j \text{ is zero on all other variables.}$$
It is easy to see that 
$$\Ker D_i=\KK[y_1,\ldots, y_m, z_1,\ldots, z_{i-1},z_{i+1},\ldots, z_p, s_1,\ldots, s_q],$$
$$\Ker E_j=\KK[y_1,\ldots, y_m, z_1,\ldots, z_p, s_1,\ldots, s_{j-1},s_{j+1},\ldots,  s_q].$$
If $q=0$ there are no $E_j$.

In \cite[Proposition~2]{G} the following proposition is proved. 
\begin{prop}\label{p1}
$ML(X)=\KK[y_1,\ldots, y_m].$
\end{prop}
\begin{proof}
If we consider all the LNDs $D_i$ and $E_j$, then intersection of their kernels equals $\KK[y_1,\ldots, y_m]$. Hence, $ML(X)\subseteq \KK[y_1,\ldots, y_m]$. 
By Proposition~\ref{rt} the variety
$$\mathbb{V}(x^3y_1^{a_1}\ldots y_m^{a_m}+z_1^{b_1}\ldots z_p^{b_p}+s_1^{c_1}\ldots s_q^{c_q}=0)$$
is rigid. Therefore, by Lemma~\ref{cm} the element $x\in\KK[X]$ is maximal. Therefore, Corollary~\ref{coco} implies the goal.
\end{proof}

In notations of this section, we have $\eta_i=\xi_{1i}$ is the following $\ZZ$-grading:
$\deg(x)=-a_i$, $\deg y_i=1$, degrees of all other variables are zero. By Lemma~\ref{nap} each $\eta_i$-homogeneous LND has positive degree.
\begin{re}\label{uuu}
If $\partial$ is a $\eta_i$-homogeneous LND, its degree not only positive, see Lemma~\ref{nap}, but even $\geq a_i$. Otherwise $\deg(\partial(x))<0$ and hence $\partial(x)=0$.
\end{re}

Now let us describe $Aut_{alg}(X)$-orbits on $X$.
\begin{theor}\label{AA}
There are the following $Aut_{alg}(X)$-orbits on $X$.
\begin{itemize}
\item All points of $X$, where all $y_1,\ldots, y_m$ are nonzero form one $Aut_{alg}(X)$-orbit $\OO$;
\item Let $\varnothing\neq M\subseteq \{1,2,\ldots, m\}$, $d=\gcd (b_1,\ldots, b_p, c_1,\ldots, c_q)$, $\varepsilon^d=1$. Then 
$$
\OO(M,\varepsilon)=\{\text{points of X }\mid y_i=0\Leftrightarrow i\in M, z_1^\frac{b_1}{d}\ldots z_p^\frac{b_p}{d}=-\varepsilon s_1^\frac{c_1}{d}\ldots s_q^\frac{c_q}{d}\neq 0\}
$$
is an $Aut_{alg}(X)$-orbit;
\item Let $\varnothing\neq M\subseteq \{1,2,\ldots, m\}$, $\varnothing\neq P\subseteq \{1,2,\ldots, p\}$, and $\varnothing\neq Q\subseteq \{1,2,\ldots, q\}$. Then
$$
\OO_1(M,P,Q)=\{\text{points of X }\mid y_i=0\Leftrightarrow i\in M, z_j=0\Leftrightarrow j\in P, s_l=0\Leftrightarrow l\in Q, x\neq 0\}
$$
is an $Aut_{alg}(X)$-orbit and 
$$
\OO_2(M,P,Q)=\{\text{points of X }\mid y_i=0\Leftrightarrow i\in M, z_j=0\Leftrightarrow j\in P, s_l=0\Leftrightarrow l\in Q, x= 0\}
$$
is an $Aut_{alg}(X)$-orbit.
\end{itemize}
\end{theor}
\begin{proof}
It is easy to see that all sets $\OO$, $\OO(M,\varepsilon)$, $\OO_1(M,P,Q)$  and $\OO_2(M,P,Q)$ cover $X$. Firstly let us check that two points from the same set $\OO$, $\OO(M,\varepsilon)$, or $\OO_i(M,P,Q)$ can be moved one to each other by $\Aut_{alg}(X)$.

Let us take $\alpha$ and $\beta$ in $\OO$. Since $y_i(\alpha)\neq 0$ and $y_i(\beta)\neq 0$, there is an element $t\in \TT$ such that $y_i(t\cdot \alpha)=y_i(\beta)$ for all $i$. Then we can apply a composition $\varphi$ of $\exp(u_i E_j)$ and $\exp (v_k D_k)$ in such a way that for all $j$ we have $z_j(\varphi(t\cdot \alpha))=z_j(\beta)$ and for all $k$ we have  $s_k(\varphi(t\cdot \alpha))=s_k(\beta)$. Note that $\varphi$ does not change $y_i$. Therefore,  $y_i(\varphi(t\cdot \alpha))=y_i(\beta)$ for all~$i$. So, points $\varphi(t\cdot \alpha)$ and $\beta$ have coinciding $y_i$, $z_j$, and $s_k$. And all $y_i$ are nonzero. Therefore, $x(\varphi(t\cdot \alpha))=x(\beta)$. Hence, $\varphi(t\cdot \alpha)=\beta$.

Now let us take $\alpha$ and $\beta$ in $\OO(M,\varepsilon)$. There is such $t\in \TT$ that $ y_i(t\cdot\alpha)=y_i(\beta)$ for $i\notin M$, $z_j(t\cdot\alpha)=z_j(\beta)$ for all $j$, and $s_k(t\cdot \alpha)=s_k(\beta)$ for all $k$. Of couse $y_i(t\cdot\alpha)=y_i(\beta)=0$ for $i\in M$. Now we can apply $\psi=\exp(uD_1)$ to obtain $x(\psi(t\cdot\alpha))=x(\beta)$. Since $y_1^{a_1}\ldots y_m^{a_m}(t\cdot\alpha)=0$, the automorphism $\psi$ does not change $z_1$. So, we obtain $\psi(t\cdot\alpha)=\beta$. 

Sets $\OO_1(M,P,Q)$ and $\OO_2(M,P,Q)$ are just $\TT$-orbits.

Let us prove that two points in different sets $\OO$, $\OO(M,\varepsilon)$, $\OO_1(M,P,Q)$  and $\OO_2(M,P,Q)$ can not be moved one to each other by $\Aut_{alg}(X)$. First of all let us note that sets $\OO$ and $\OO(M,\varepsilon)$ consist of regular points, while sets $\OO_1(M,P,Q)$ and $\OO_2(M,P,Q)$ consist of singular points.  Each function $y_i$ is $\SAut(X)$-invariant and semi-invariant under every algebraic action of a one-dimensional torus, see Lemma~\ref{leml}. Hence, each $y_i$ is $\Aut_{alg}(X)$-semi-invariant. Therefore, action of an element from $\Aut_{alg}(X)$ can not change the set $M$. These two assertion imply that $\OO$ and $\OO(M)=\bigcup\OO(M,\varepsilon)$ are $\Aut_{alg}(X)$-invariant. If $d=1$, the set $\OO(M)$ is a unique $\Aut_{alg}$-orbit. If $d>1$, the variety $\OO(M)$ is reducible and $\OO(M,\varepsilon)$ are its irreducible components. Since $\Aut_{alg}$ is generated by connected subgroups, it can not nontrivially permute irreducible components of $\OO(M)$.

If $p+q>2$, the set of singular points $X^{sing}$ is reducible. Irreducible components of $X^{sing}$ has the form $V_{i,j,k}=\{y_i=z_j=s_k=0\}$. So, the triple of sets $M,P,Q$ depends only on the information to which components of $X^{sing}$ belongs our point. By Proposition~\ref{linn}, $\Aut_{alg}(X)$ do not change the set of components containing a point. This implies that the union $\OO_1(M,P,Q)\cup\OO_2(M,P,Q)$ is  $\Aut_{alg}(X)$-invariant. In other words  $\OO_1(M,P,Q)$ and $\OO_2(M,P,Q)$ are linked, but they are not linked to any other $\TT$-orbits.

Now let us prove that we can not glue  $\OO_1(M,P,Q)$ and $\OO_2(M,P,Q)$ by $\Aut_{alg}(X)$. Fix $\alpha\in \OO_2(M,P,Q)$. Suppose $\partial$ is an LND of $\KK[X]$ such that $\partial(x)(\alpha)\neq 0$. Remind that $y_1,\ldots ,y_m\in ML(X)$. Therefore, 
$$\partial(xy_1^{a_1}\ldots y_m^{a_m})=\partial(x)y_1^{a_1}\ldots y_m^{a_m}.$$
By Remark~\ref{uuu}, the decomposition of $\partial$ onto $\eta_i$-homogeneous components has the form $\partial=\sum_{j=l}^k\partial_j$, where $l\geq a_i$. Therefore, $\deg \partial(z_u)\geq a_j$. That is $y_i^{a_i}\mid \partial(z_u)$. Similarly, 
$
(y_1^{a_1}\ldots y_m^{a_m})\mid \partial(z_u)
$
and 
$
(y_1^{a_1}\ldots y_m^{a_m})\mid \partial(s_v).
$
 Denote $\partial(z_u)=f_u\cdot (y_1^{a_1}\ldots y_m^{a_m})$, $\partial(s_v)=g_v\cdot (y_1^{a_1}\ldots y_m^{a_m})$.
We obtain

\begin{multline*}
0=\partial(xy_1^{a_1}\ldots y_m^{a_m})+\partial(z_1^{b_1}\ldots z_p^{b^p})+\partial(s_1^{c_1}\ldots s_q^{c^q})=\\
=y_1^{a_1}\ldots y_m^{a_m}\left(\partial(x)+\sum_{u=1}^p b_uz_1^{b_1}\ldots z_u^{b_u-1}\ldots z_p^{b_p}f_u+\sum_{v=1}^q c_vs_1^{c_1}\ldots s_v^{c_v-1}\ldots s_q^{c_q}g_v\right)
\end{multline*}
Therefore, 
$$
\partial(x)=-\sum_{u=1}^p b_uz_1^{b_1}\ldots z_u^{b_u-1}\ldots z_p^{b_p}f_u-\sum_{v=1}^q c_vs_1^{c_1}\ldots s_v^{c_v-1}\ldots s_q^{c_q}g_v.
$$
So, $\partial(x)(\alpha)=0$. Hence, $\OO_2(M,P,Q)$ is $\SAut(X)$-invariant.

Let $\delta$ be a semisimple derivation of $X$. Let us consider $\ZZ^m$-grading $\eta=\eta_1\oplus\ldots\oplus \eta_m$. Consider the decomposition of $\delta$ onto $\eta$-homogeneous components $\delta=\sum\limits_{j_1,\ldots,j_m}\delta_{j_1,\ldots,j_m}$.  By Lemma~\ref{leml}(c) $\delta(y_i)=\lambda_i y_i$. Then $\delta_{0,\ldots,0}(y_i)=\lambda_i y_i$ and $\delta_{j_1,\ldots,j_m}(y_i)=0$ if at least one of $j_i$ is nonzero. Let us prove for each $j_1,\ldots,j_m$ that $\delta_{j_1,\ldots,j_m}(x)(\alpha)=0$. If al least one $j_i<a_i$, then $x\mid\delta_{j_1,\ldots,j_m}(x)$. That is $\delta_{j_1,\ldots,j_m}(x)(\alpha)=0$. If If $j_i\geq a_i$ for all $i$, then as above 
$
(y_1^{a_1}\ldots y_m^{a_m})\mid\delta_{j_1,\ldots,j_m}(z_u)
$
and 
$
(y_1^{a_1}\ldots y_m^{a_m})\mid \delta_{j_1,\ldots,j_m}(s_v).
$
Therefore, 
$$
\delta_{j_1,\ldots,j_m}(x)=-\sum_{u=1}^p b_uz_1^{b_1}\ldots z_u^{b_u-1}\ldots z_p^{b_p}\widetilde{f}_u-\sum_{v=1}^q c_vs_1^{c_1}\ldots s_v^{c_v-1}\ldots s_q^{c_q}\widetilde{g}_v.
$$
 That is $\delta_{j_1,\ldots,j_m}(x)(\alpha)=0$. Since it holds for every homogeneous component, we have $\delta(x)(\alpha)=0$. Hence, we can not move $\alpha$ to $\OO_1(M,P,Q)$ by one-dimensional torus. Therefore we can not glue  $\OO_1(M,P,Q)$ and $\OO_2(M,P,Q)$ by $\Aut_{alg}(X)$.
\end{proof}
\begin{re}
If $X$ is a trinomial hupersurface with a free term then $X$ is smooth. So, we have only $\Aut_{alg}$-orbits $\OO$ and $\OO(M,\varepsilon)$.
\end{re}
\begin{re}
Note, that all $\Aut_{alg}(X)$-orbits turned to be orbits of the group generated by $\Ga$-subgroups $\{\exp(uD_i)|u\in\KK\}$, $\{\exp(vE_j)|v\in\KK\}$ and $\TT$.
\end{re}
\begin{re}
The number of $\Aut_{alg}(X)$-orbits is finite. It equals 
$$1+(2^{m}-1)d+2(2^m-1)(2^p-1)(2^{q}-1).$$
\end{re}
\begin{re}
It is easy to see that $\OO$ is an open orbit, i.e. $\dim \OO=\dim X=m+p+q$. Dimension of $\OO(M,\varepsilon)$ equals $\dim X-|M|=m+p+q-|M|$. Dimension of $\OO_1(M,P,Q)$ equals $\dim X+1-|M|-|P|-|Q|=m+p+q+1-|M|-|P|-|Q|$. And $\dim \OO_2(M,P,Q)$ equals $\dim X-|M|-|P|-|Q|=m+p+q-|M|-|P|-|Q|$.
\end{re}

Also we can describe $\SAut(X)$-orbits.

\begin{theor}
$\SAut(X)$-orbits on $X$ have the following form. 
\begin{itemize}
\item $\SSS(d_1,\ldots, d_m)=\{\alpha\in X \mid y_1(\alpha)=d_1,\ldots, y_m(\alpha)=d_m\}$, all $d_i\neq 0$. 

Each $\SSS(d_1,\ldots, d_m)$ is isomorphic to $\KK^{p+q}$.
\item $\LLL(d_1,\ldots, d_m, r_1,\ldots, r_p, t_1,\ldots, t_q)=$

$= \{\alpha\in X \mid y_i(\alpha)=d_i, z_j(\alpha)=r_j, s_k(\alpha)=t_k\text{ for all } i,j,k\},$ 

there exists $d_l=0$, all $r_j$ and all $t_k$ are nonzero, and $r_1^{b_1}\ldots r_p^{b_p}=-t_1^{c_1}\ldots t_q^{c_q}$. 

Each $\LLL(d_1,\ldots, d_m, r_1,\ldots, r_p, t_1,\ldots, t_q)$ is isomorphic to a line. 
\item Each singular point is $\SAut(X)$-stable.
\end{itemize}
\end{theor}
\begin{proof}
Consider $\alpha, \beta\in\SSS(d_1,\ldots, d_m)$. Since, $y_i\in ML(X)$, the orbit $\SAut(X)\alpha$ is contained in $\SSS(d_1,\ldots, d_m)$. But we can apply a composition $\varphi$ of $\exp(u_i E_j)$ and $\exp (v_k D_k)$ in such a way to obtain $z_j(\varphi(\alpha))=z_j(\beta)$ and $s_k(\varphi(\alpha))=s_k(\beta)$ for all $j$ and $k$. Then $\alpha=\beta$. I.e. $\SSS(d_1,\ldots, d_m)$ is a unique $\SAut(X)$-orbit.

Now let us consider $\alpha, \beta\in\LLL(d_1,\ldots, d_m, r_1,\ldots, r_p, t_1,\ldots, t_q)$. We can move $\alpha$ to $\beta$ by $\exp(u D_1)$. There exists $d_l=0$. Since $y_i\in ML(X)$, each $\psi\in \SAut(X)$ induces an automorphism of $X\cap \{y_1=d_1,\ldots, y_m=d_m\}\cong \KK\times Z$, where $Z$ is the rigid toric (may be not irreducible, may be not normal) variety
$$
\VV(z_1^{b_1}\ldots z_p^{b_p}+t_1^{c_1}\ldots t_q^{c_q}).
$$ 
The fact that $Z$ is rigid follows from rigidity of the variety 
$$
\VV(u^3+z_1^{b_1}\ldots z_p^{b_p}+t_1^{c_1}\ldots t_q^{c_q}).
$$
The last one is rigid by Proposition~\ref{rt}.
Since $Z$ is rigid, $Z\times \KK$ is semi-rigid, see~\cite[Theorem~2.24]{Fr}. Therefore, $\SAut(X)$-orbits on $Y$ are lines.

Each $\Aut_{alg}(X)$-orbit consisting of singular points is a $\TT$-orbit. Hence, it is isomorphic to a torus. Therefore, it does not admit a nontrivial action of the additive group $(\KK,+)$. So, $\SAut(X)$-orbits of singular points are points.
\end{proof}

\begin{re}
Typical $\SAut(X)$-orbits can be separated by functions from $ML(X)$.
\end{re}

Now we are able to describe $\Aut(X)$-orbits on $X$. They can be obtained by gluing of $\Aut_{alg}(X)$-orbits. If $d=\gcd(b_1,\ldots, b_p,c_1,\ldots, c_q)$, then $\HH=\TT\times\GG$, where $\GG$ is a cyclic group $\CCC_d$, that permutes $\OO(M,\varepsilon)$ for different $\varepsilon$.   The group $\Aut_{alg}(X)$ is a normal subgroup of $\Aut(X)$. Therefore, each automorphism induces a permutation of $\Aut_{alg}(X)$-orbits.

\begin{prop}\label{mmm} 
Two $\Aut_{alg}(X)$-orbits $O'$ and $O''$ are contained in one $\Aut(X)$-orbit if and only if there is element of $\SS(X)\rightthreetimes \GG$ which maps $O'$ to $O''$.
\end{prop}
\begin{proof}
Since $\SS(X)\rightthreetimes \GG$ consists of automorphisms, if $O'$ can be moved to $O''$ by an element of $\SS(X)\rightthreetimes \GG$, then $O'$ and $O''$ are in one $\Aut(X)$-orbit.

Conversely, assume there exists $\varphi\in \Aut(X)$ such that $\varphi(O')=O''$.  Functions $y_1,\ldots, y_m$ are indecomposable in $ML(X)$ and  semi-invariant under every $\KK^\times$-action. By Corollary~\ref{coco}, $\varphi(y_i)=\lambda y_j$. Let us prove that it is possible only if $a_i=a_j$.
\begin{lem}\label{ll}
Denote by $U\subseteq \KK[X]$ the set of nonmaximal elements. Let $\partial$ be an LND of $\KK[X]$. Then $\partial(U)\subseteq (y_1^{a_1}y_2^{a_2}\ldots y_m^{a_m})$. 
\end{lem}
\begin{proof}
If we consider $\eta_i$-homogeneous components of $\partial=\sum_{i=l}^k\partial_i$, then $l\geq a_i$ by Remark~\ref{uuu}. Then $\deg\partial(z_j)\geq a_i$ and  $\deg\partial(s_k)\geq a_i$. Hence, $\partial(z_j)$ and $\partial(s_k)$ is divisible by $y_i^{a_i}$. Since it is true for all $i$, we have $\partial(z_j),\partial(s_k)\in (y_1^{a_1}\ldots y_m^{a_m})$. So, we have that $\delta(\KK[y_1,\ldots,y_m, z_1,\ldots, z_p, s_1,\ldots, s_q])\subseteq (y_1^{a_1}\ldots y_m^{a_m})$. Now note that $x$ is a maximal element of $\KK[X]$. Each element of $\KK[X]\setminus\KK[y_1,\ldots,y_m, z_1,\ldots, z_p, s_1,\ldots, s_q])$ is divisible by $x$. Hence, it is maximal.
\end{proof}
Note that $y_1^{a_1}\ldots y_m^{a_m}= D_1(-z_1)\in D_1(U)$. And $z_1$ is not a maximal element. Hence if $\sigma\in \mathrm{S}_m$, and there is $i$ such that $a_{i}\neq a_{\sigma(i)}$, then 
$$\Im D_1(U)\not\subseteq (y_1^{a_{\sigma(1)}}\ldots  y_m^{a_{\sigma(m)}}).$$ 
Therefore, $\varphi(y_i)$ can not be equal to $\lambda y_i$ if $a_i\neq a_j$. But we can permute $y_i$ with equal powers~$a_i$ by $\SS(X)$ and multiply them by $\lambda$ by $\TT$-action. So, we can take such a composition $\psi=\rho\circ t\circ\varphi$, where $\rho$ is an element of $\SS(X)$ and $t$ is an element of $\TT$,  that $\psi(y_i)=y_i$ for all $i$. Denote $O'''=\psi(O')$. If we prove that there is $\zeta\in\SS(X)\rightthreetimes \GG$ such that $\zeta(O')=O'''$. Then $\rho^{-1}\circ \zeta(O')=O''$. If $O'=\OO$, then $\dim O'''=\dim O'=\dim X$, Hence, $O'''=O'$. If $O'=\OO(M,\varepsilon)$, then it consists of regular points. Therefore, $O'''$ consists of regular points. Hence, $O'''=\OO(\widetilde{M},\widetilde{\varepsilon})$. Since $\psi(y_i)=y_i$, we have $\widetilde{M}=M$ and we can map $O'$ to $O'''$ by $\GG$. Let $O'\subseteq W(M,P,Q)=\OO_1(M,P,Q)\cup \OO_2(M,P,Q)$. Then $\psi$ induces automorphism of 
$$Z(M)=X\cap \{y_i=0|i\in M, y_j=1, j\notin M\}\cong \KK\times Y,$$
where
$
Y=\VV(z_1^{b_1}\ldots z_p^{b_p}+t_1^{c_1}\ldots t_q^{c_q})
$ 
is a rigid toric variety. This automorphism can be restricted to an automorphism $\widehat{\psi}$ of $ML(\KK\times Y)=\KK[Y]$. The permutation of irreducible components $V_{i,j,k}=\{y_i=z_j=s_k=0\}$ of $X^{sing}$ given by $\psi$ corresponds to the permutation of irreducible components $V_{j,k}=\{z_j=s_k=0\}$ of $Y$ by $\widehat{\psi}$. But $Y$ is a rigid variety, hence, it has a unique torus in its automorphism group. Toric variety $Y$ corresponds to its weight monoid. Automorphism $\widehat{\psi}$ corresponds to an automorphism of this monoid. We can consider the cone generated by this monoid. Then if $p>1$ and $q>1$, weights of $z_j$ and $s_k$ belong to the extremal rays of the cone. I.e. these weights a the shortest vectors of weight monoid on these rays. We have the unique relation between the weights: $b_1w(z_1)+\ldots+b_pw(z_p)=c_1w(s_1)+\ldots+c_qw(s_q)$. So, we can obtain numbers $b_i$ and $c_i$ from the weight monoid. Therefore, $\widehat{\psi}$ maps each variable to $\lambda$ multiplied by a variable with the same power in the equation. That is we can not permute irreducible components of $Y^{sing}$ in a way different to permutation given by an element of $\SS(X)$. If $p>1$ and $q=1$, then only weights of $z_j$ correspond to extremal rays of the cone. And the weight of $s_1$ is a the shortest vector in weight monoid which is not in monoid generated by all $w(z_i)$. That is weight of $s_1$ and weights of $z_{ij}$ are defined in terms of weight monoid. And we can obtain numbers $b_1, \ldots, b_p, c_1$ from the relatins between primitive vectors on extremal rays and the shortest vector that is not generated by these vectors. Hence, $\widehat{\psi}$ maps each $z_j$ to $z_k$ with the same power in the equation. If $p=q=1$, then we have unique irreducible component of $Y^{sing}$.

Thus, we have proved that permutation of $V_{j,k}$ given by $\widehat{\psi}$ can be given by an element of $\tau\in\SS(X)$. Since $\psi$ preserves every $y_i$, we obtain that permutation of $V_{i,j,k}$ given by $\psi$ can be given by an element of $\tau$. Therefore, $\tau$ maps $W(M,P,Q)$ to $\psi(W(M,P,Q))$. But $\Aut_{alg}$-orbits $\OO_1(M,P,Q)$ and $\OO_2(M,P,Q)$ has different dimensions. That is $\tau(\OO_1(M,P,Q))=\psi(\OO_1(M,P,Q))$ and $\tau(\OO_2(M,P,Q))=\psi(\OO_2(M,P,Q))$.
\end{proof}

Since $\GG$ glues only $\OO(M,\varepsilon)$ for different $\varepsilon$ into $\OO(M)$ and does not change  $\OO$, $\OO_1(M,P,Q)$, and $\OO_2(M,P,Q)$, we have the following theorem.

\begin{theor} \label{maintheor}
Orbits of automorphism group on $X$ are the sets obtained by gluing by $\SS(X)$ the following sets: $\OO$, $\OO(M)$, $\OO_1(M,P,Q)$, and $\OO_2(M,P,Q)$.
\end{theor}

\begin{ex}
Let $X=\VV(xy_1^2y_2^2+z^3+s^3).$ We have the following list of $\Aut(X)$-orbits.
\begin{itemize}
\item $\OO=\{y_1,y_2\neq 0\}$;
\item $\OO(\{1\})\cup\OO(\{2\})=\{y_1=0;y_2,z,s\neq 0\}\cup\{y_2=0;y_1,z,s\neq 0\}$;
\item $\OO(\{1,2\})=\{y_1=y_2=0; z,s\neq 0\}$;
\item $\OO_1(\{1\}, \{1\},\{1\})\cup\OO_1(\{2\}, \{1\},\{1\})=$

$=\{y_1=z=s=0;x,y_2\neq 0\}\cup\{y_2=z=s=0;x,y_1\neq 0\}$
\item $\OO_2(\{1\}, \{1\},\{1\})\cup\OO_2(\{2\}, \{1\},\{1\})=$

$=\{x=y_1=z=s=0;y_2\neq 0\}\cup\{x=y_2=z=s=0;y_1\neq 0\}$
\item $\OO_1(\{1,2\},\{1\},\{1\})=\{y_1=y_2=z=s=0; x\neq 0\};$
\item $\OO_2(\{1,2\},\{1\},\{1\})=\{x=y_1=y_2=z=s=0\}.$
\end{itemize}
\end{ex}

\begin{ex}
Let $X=\VV(xy_1^3y_2^3+z_1^3z_2^3+1).$ We have the following list of $\Aut(X)$-orbits.
\begin{itemize}
\item $\OO=\{y_1,y_2\neq 0\}$;
\item $\OO(\{1\})\cup\OO(\{2\})=\{y_1=0,y_2\neq 0\}\cup\{y_2=0,y_1\neq 0\}$;
\item $\OO(\{1,2\})=\{y_1=y_2=0\}$.
\end{itemize}
\end{ex}

\section{Nonflexible varieties with a vairiable with power one}

In this section we consider a trinomial hypersurface for which we do not know, how $ML(X)$ looks like. We have a conjection about $ML(X)$. If it is true, we describe $\Aut_{alg}(X)$-orbits.
Let 
$$
X=\mathbb{V}(x_1x_2\ldots x_ky_1^{a_1}\ldots y_m^{a_m}+z_1^{b_1}\ldots z_p^{b_p}+s_1^{c_1}\ldots s_q^{c_q}), a_i, bj, c_l>1, m\geq 1, p\geq 1, q\geq 0, k>1.
$$
We assume $X$ to be nonflexible, i.e. not of type $H_4$.

Now we have the following LNDs:

$$D_{ij}(x_i)=\frac{\partial (z_1^{b_1}\ldots z_p^{b_p})}{\partial z_j}, D_{ij}(z_j)=-\frac{x_1\ldots x_ky_1^{a_1}\ldots y_m^{a_m}}{\partial x_i}, D_{ij} \text{ is zero on all other variables;}$$
$$E_{ij}j(x_i)=\frac{\partial (s_1^{c_1}\ldots s_q^{c_q})}{\partial s_j}, E_{ij}(s_j)=-\frac{x_1\ldots x_ky_1^{a_1}\ldots y_m^{a_m}}{\partial x_i}, E_{ij} \text{ is zero on all other variables.}$$

Intersection of kernels of all $D_{ij}$ and $E_{ij}$ equals $\KK[y_1,\ldots, y_m]$. So, let us state a conjection.
\begin{conj}
ML(X)=$\KK[y_1,\ldots, y_m]$.
\end{conj}

It is sufficient to assume that $\KK[y_1,\ldots, y_m]$ is a $\Aut(\KK[X])$-invariant subring of $\KK[X]$.

\begin{lem}\label{si}
Assume $\KK[y_1,\ldots, y_m]$ is a $\Aut(\KK[X])$-invariant subring of $\KK[X]$. Then for each $1\leq i\leq k$ and $\varphi\in\Aut(X)$ we have $\varphi(y_i)=\lambda y_j$, $\lambda\in \KK^\times$.
\end{lem}

\begin{proof}
Each irreducible component of the singular locus $X^{sing}$ has either the form $\VV(x_i, x_j, z_u, s_v)$ or the form  $\VV(y_i, z_u, s_v)$. This components have different dimensions. Therefore under an automorphism $\varphi$ a component $\VV(y_i, z_u, s_v)$ maps to a component $\VV(y_j, z_{u'}, s_{v'})$. Hence, $\varphi(y_i)\in (y_j, z_{u'}, s_{v'})$. But $\varphi(y_i)\in \KK[y_1,\ldots, y_m]$. That is $\varphi(y_i)=\lambda y_j$.
\end{proof}

\begin{lem} In assumptions of Lemma~\ref{si}, we have $ML(X)=\KK[y_1,\ldots, y_m]$.
\end{lem}
\begin{proof}
Lemma~\ref{si} implies that $ML(X)\supseteq \KK[y_1,\ldots, y_m]$. But intersection of all $D_{ij}$ and $E_{ij}$ equals $\KK[y_1,\ldots, y_m]$.
\end{proof}

\begin{theor}\label{dd}
Assume $\KK[y_1,\ldots, y_m]$ is a $\Aut(\KK[X])$-invariant subring of $\KK[X]$. Then there are the following $\Aut_{alg}$-orbits on $X$.

\begin{itemize}
\item All regular points of $X$, where all $y_1,\ldots, y_m$ are nonzero form one $\Aut_{alg}$-orbit $\OO$;
\item Let $\varnothing\neq M\subset \{1,\ldots, m\}$. Denote $d=\gcd(b_1,\ldots, b_p, c_1,\ldots, c_q)$. Let $\varepsilon^d=1$. Then 
$$
\OO(M,\varepsilon)=\{\text{points of X }\mid y_j=0\Leftrightarrow j\in M, z_1^\frac{b_1}{d}\ldots z_p^\frac{b_p}{d}=-\varepsilon s_1^\frac{c_1}{d}\ldots s_q^\frac{c_q}{d}\neq 0\}
$$
is an $Aut_{alg}(X)$-orbit;
\item Let $K\subset \{1,\ldots, k\}$, $M\subseteq \{1,2,\ldots, m\}$, $2|M|+|K|\geq 2$, $\varnothing\neq P\subseteq \{1,2,\ldots, p\}$, and $\varnothing\neq Q\subseteq \{1,2,\ldots, q\}$. Then
\begin{multline*}
\OO(K,M,P,Q)=\\
=\{\text{points of X }\mid x_i=0\Leftrightarrow i\in K, y_j=0\Leftrightarrow j\in M, z_l=0\Leftrightarrow l\in P, s_r=0\Leftrightarrow r\in Q\}
\end{multline*}
is an $Aut_{alg}(X)$-orbit.
\end{itemize} 
\end{theor}
\begin{proof}
It easy to see that all this subsets do not intersect and cover all $X$. Firstly let us check that two points from the same set $\OO$, $\OO(M,\varepsilon)$, or $\OO(K,M,P,Q)$ can be moved one to each other by $\Aut_{alg}(X)$. It is easy to see that $\OO(K,M,P,Q)$ are $\TT$-orbits. If $\alpha,\beta\in \OO(M,\varepsilon)$, then applying $t\in \TT$ we can obtain $y_i(t\cdot \alpha)=y_i(\beta)$, $z_j(t\cdot \alpha)=z_j(\beta)$, $s_l(t\cdot\alpha)=s_l(\beta)$. Then applying
$\exp(uD_{i1})$ we can move $x_1,\ldots, x_k$. 

Let $\alpha,\beta\in \OO$. Then either there is $j$ such that all $x_i(\beta)\neq 0$ for $i\neq j$, or $z_1(\beta),\ldots, z_p(\beta)\neq 0$. In the second case we can apply $\exp(vD_{i1})$ to obtain $x_i\neq 0$. So we can assume $x_i(\beta)\neq0$ for $i\geq 2$.  We can apply  $t\in \TT$  to obtain $y_i(t\cdot \alpha)=y_i(\beta)$, $x_j(t\cdot \alpha)=x_j(\beta)$, $j\geq 2$. Then by $\exp(uD_{1i})$ and $\exp(uE_{1i})$ we can make $z_i$ and $s_j$ coinciding on $\beta$ and image of $\alpha$. Then $x_1$ also coincides.

Now we are to prove that these sets can not be glued by $\Aut_{alg}(X)$. By Lemma~\ref{si}, each $y_i$ is $\Aut(X)$-semi-invariant. Therefore, we can not change the set $M$ by automorphisms. The set of irreducible components containing a point can not be changed by $\Aut_{alg}(X)$. Therefore, we only are to prove that 
$\OO(M,\varepsilon)$ can not be glued by $\Aut_{alg}(X)$ for different $\varepsilon$ and that $\OO(\{x_r\},M,P,Q)$ can not be glued with $\OO(\varnothing, M,P,Q)$, where $M\neq\varnothing$. Since $\Aut_{alg}(X)$ is generated by connected subgroups, it can not glue irreducible components of the invariant subset $\OO(M)$. 
This components are  $\OO(M,\varepsilon)$. Components $\OO(\{x_r\},M,P,Q)$ and $\OO(\varnothing, M,P,Q)$ exist only if $X$ is a trinomial hypersurface without free term. Let us fix $y_i\in M$, $z_j\in P$ and $s_l\in Q$. Consider the $\ZZ$-grading $\rho$:
$$
\deg_{\rho}(y_i)=b_jc_l, \deg_{\rho}(z_j)=a_ic_l, \deg_{\rho}(s_l)=a_ib_j, \text{ degrees of other variables are zero.}
$$
\begin{lem}\label{spl}
Let $\partial$ be a nonzero $\rho$-homogeneous LND. Then $\deg_{\rho}(\partial)> 0$. 
\end{lem}
\begin{proof}
Let us consider gradings $\zeta_1,\ldots, \zeta_k$, where 
$$\deg_{\zeta_t}(x_t)=-a_i, \deg_{\zeta_t}(y_i)=1, \text{ degrees of other variables are zero.} $$
Consider the $\ZZ^{k+1}$-grading $\rho\oplus\zeta_1\oplus\ldots\oplus\zeta_k$.
Suppose $\deg_{\rho}\partial\leq0$, then we can decompose $\partial$ onto the sum of $\rho\oplus\zeta_1\oplus\ldots\oplus\zeta_k$-homogeneous components, some of which should be nonzero $\rho\oplus\zeta_1\oplus\ldots\oplus\zeta_k$-homogeneous LNDs with nonpositive $\rho$-degrees. Suppose,  $\delta$ is such a component. If 
$\delta(x_1)=\ldots=\delta(x_k)=0$, then $\delta$ induces an LND of 
$$\VV(x_1^3\ldots x_k^3y_1^{a_1}\ldots y_m^{a_m}+z_1^{b_1}\ldots z_p^{b_p}+s_1^{c_1}\ldots s_q^{c_q}),$$ 
which is rigid by Proposition~\ref{rt}. So, there exists $u$ such that $\delta(x_u)\neq 0$. That is $\deg_{\zeta_u}(\delta)>0$. Hence, $y_i\mid\delta(x_t)$ for all $t\neq u$. That is either $\deg_{\rho}>0$, or $\delta(x_t)=0$ for all $t\neq u$. If $\delta(x_t)=0$ for all $t\neq u$ we can consider the field $\LL=\overline{\KK(x_1,\ldots, x_{u-1}, x_{u+1},\ldots, x_k)}$. Then $\delta$ induces an LND of 
$$X(\LL)\cong\VV(x_uy_1^{a_1}\ldots y_m^{a_m}+z_1^{b_1}\ldots z_p^{b_p}+s_1^{c_1}\ldots s_q^{c_q}).$$
As we have proved in Lemma~\ref{ll}, $y_1^{a_1}\ldots y_m^{a_m}$ devids all images $\delta(z_v)$ and $\delta(s_w)$. This implies $\deg_\rho(\delta)>0$.
Lemma is proved.
\end{proof}

Now, let $\partial$ be an LND and $\alpha\in \OO(\{x_r\},M,P,Q)$. Consider the decomposition of $\partial$ onto $\rho$-homogeneous components: $\partial=\sum_{i=l}^k\partial_i$. We have $l>0$. Hence, $\deg_\rho(\partial_i(x_r))>0$. Therefore, $\partial(x_r)(\alpha)=0$. Hence, we can not move by $\SAut(X)$ the point $\alpha$ outside $\OO(\{x_r\},M,P,Q)$. 

Let us consider a semisimple derivation $\delta$.  Consider $\ZZ^{2}$-grading $\rho\oplus\zeta_r$. Let $\delta=\sum\delta_{u_1u_2}$ is
 its decomposition onto homogeneous components. If $u_1<0$ then we have a summand of negative $\rho$-degree in the decomposition of $\delta$ onto $\rho$-homogeneous components. This contradicts to Lemmas~\ref{fl} and \ref{spl}. Therefore, $u_1\geq 0$. If $u_1>0$, then $\delta_{u_1u_2}(x_r)(\alpha)=0$. Suppose $u_1=0$. 
 If $u_2<a_2$, then $x_r\mid \delta_{u_1u_2}(x_r)$. Hence, $\delta_{u_1u_2}(x_r)(\alpha)=0$. 
 If $u_2\geq a_2$, then $y_i^{a_i}\mid\delta_{u_1u_2}(z_j)$ , $y_i^{a_i}\mid\delta_{u_1u_2}(s_l)$ for all $j,l$, and $y_i^{a_i}\mid \delta_{u_1u_2}(x_v)$, $v\neq r$. Since $u_1=0$ this implies $\delta_{u_1u_2}(z_j)=\delta_{u_1u_2}(s_l)= \delta_{u_1u_2}(x_v)=0$. We know that $\delta(y_j)=\lambda_j y_j$ for all $j$. Hence, if $u_2\neq0$, then $\delta_{0u_2}(y_j)=0$. Therefore, $\delta_{0u_2}=0$. So, in all cases $\delta_{u_1u_2}(x_r)(\alpha)=0$.  Thus, $\delta(x_r)(\alpha)=0$, i.e. we can not move by $\Aut_{alg}(X)$ the point $\alpha$ outside $\OO(\{x_r\},M,P,Q)$.

\end{proof}

\end{document}